\documentclass[a4paper,12pt,draft]{article}
\usepackage{amsmath,amssymb}
\usepackage{amsthm}
\usepackage[dvipdfmx]{graphicx}
\usepackage{mathrsfs}
\usepackage[active]{srcltx}
\usepackage{amscd}
\usepackage{cases}
\usepackage[noadjust]{cite}
\usepackage[usenames]{color}
\usepackage{titlesec}
\usepackage{cancel}
\titleformat*{\section}{\large\bfseries}
\titleformat*{\subsection}{\normalsize\bfseries}
 \newtheorem{theorem}{Theorem}[section]
 \newtheorem{proposition}[theorem]{Proposition}
 
 \newtheorem{lemma}[theorem]{Lemma}
 \newtheorem{corollary}[theorem]{Corollary}
\theoremstyle{definition}
 \newtheorem{definition}[theorem]{Definition}
 
 \newtheorem{example}[theorem]{Example}
\numberwithin{equation}{section}

\newcommand{\R}{\boldsymbol{R}}

\newcommand{\A}{\mathcal{A}}

\newcommand{\OO}{\mathcal{O}}
\newcommand{\W}{\mathcal{W}}
\newcommand{\pmt}[1]{{\begin{pmatrix} #1  \end{pmatrix}}}

\newcommand{\e}{\boldsymbol{e}}
\newcommand{\inner}[2]{\left\langle{#1},{#2}\right\rangle}

\newcommand{\rank}{\operatorname{rank}}

\renewcommand{\phi}{\varphi}

\newcommand{\image}{\operatorname{image}}
\renewcommand{\Gamma}{\varGamma}
\newcommand{\ep}{\varepsilon}

\newcommand{\gr}{\operatorname{Gr_2(4)}}
\title{Geometry of pseudo-non-degenerate two-ruled hypersurfaces}
\author{Junzhen Li and Kentaro Saji}
\date{\today}
\begin{document}
\maketitle
\footnote[0]{
Partly supported by the
JSPS KAKENHI Grant Numbers 25K07001, 22KK0034.}

\begin{abstract}
	We investigate the singularities of
	two-ruled hypersurfaces in the Euclidean four-space.
	By considering the points that minimize the
	distance between adjacent rulings,
	we obtain a characterization the striction curve.
	We introduce the notion of pseudo-non-degenerate two-ruled hypersurfaces
	and examine their fundamental properties.
	We show that two-ruled hypersurfaces constructed from
	a curve equipped with a Frenet-type frame,
	via height functions, are pseudo-non-degenerate.
	Furthermore, we study
	properties of the original curve through the
	striction curves and the
	singularities of pseudo-non-degenerate two-ruled hypersurfaces
	constructed in this manner.

\end{abstract}

\section{Introduction}
In this paper, we study one-parameter families of
planes in the Euclidean four-space $\R^4$, called
two-ruled hypersurfaces.
Ruled surfaces in $\R^3$ are a classical subject in differential
geometry and have recently attracted attention not only within
differential geometry but also in 
several other fields. See, for example,
\cite{hhnuy,izta,izoh}
and
\cite{flpo,hnsuy,li,pxwd,ppr}.
Two-ruled hypersurfaces naturally generalize
ruled surfaces to higher-dimensional spaces, and they have been studied
by many authors \cite{2ruledhypersurface,md,sing2ruledhypersurface}.
In \cite{sing2ruledhypersurface}, a notion of
non-degeneracy for two-ruled hypersurfaces is introduced.
Under this condition,
the striction surface and the striction curve, 
which are an analogue of the striction curve of a ruled surface, 
are naturally defined, since
the singular points lie on that curve.
In this paper, by characterizing the striction curve
as the locus of points that
the minimize, in an infinitesimal sense, the
distance between adjacent rulings,
we show that the striction curve in the sense of \cite{sing2ruledhypersurface}
also has the same meaning.
We introduce another notion, pseudo-non-degeneracy, for 
two-ruled hypersurfaces,
and define the
striction surface and second striction curve
on pseudo-non-degenerate two-ruled hypersurfaces.
Given a curve in $\R^4$, we take a Frenet-type frame, 
and define
a height function with respect to one of the frame vectors.
The discriminant set of this height function
yields a pseudo-non-degenerate two-ruled hypersurface.
By studying the striction surfaces and singularities of these 
hypersurfaces, we obtain information about the geometry
of the original curve.

\section{Two-ruled hypersurfaces}
In this section, we introduce two-ruled hypersurfaces
and their basic properties.
For further details, see \cite{2ruledhypersurface,sing2ruledhypersurface}.
\subsection{Basic notions and definitions}\label{sec:basic}
Let $\gamma:(\R,0) \to (\R^4,0)$
and
$P:(\R,0) \to \gr$ be curve-germs, where $\gr$ is
a Grassmann manifold.
Taking a basis field $X,Y:(\R,0) \to \R^4$ of $P$,
the one-parameter family of hypersurfaces defined by
$$
f(t,s,r) = \gamma(t) + s X(t) + r Y(t)
$$
for $t\in(\R,0)$ and $s,r\in\R$
is called a \emph{two-ruled hypersurface-germ},
or simply a \emph{two-ruled hypersurface}.
Since $O=\{(t,s,r)\in\R^3\,|\,t=0\}$ is a closed set,
one can naturally consider map-germs along $O$,
and we denote them by 
$f:(\R^3,O)\to\R^4$ or we write
$f:(\R^3,(0,s,r))\to\R^4$
to explicitly denote the variable.
Because we work with germs at $0$, the expression 
``for all $t \in (\R,0)$'' means that 
there exists an open neighborhood $I \subset \R$ of $0$ 
such that the statement holds for all $t \in I$.
We call $\gamma$ the {\it base curve}, and the two vectors $X$ and $Y$ 
the {\it director curves}. 
The planes $P$ are called {\it rulings}.
This definition generalizes the notion of ruled surfaces in $\R^3$.
The properties of two-ruled hypersurfaces, 
including their generic singularities, 
have been studied in \cite{2ruledhypersurface,sing2ruledhypersurface}.
For any two-ruled hypersurface-germ 
$f=\gamma(t) + s X(t) + r Y(t):(\R^3,(0,s,r))\to\R^4$, 
we can choose director curves $X$ and $Y$ such that $|X|=|Y|=1$,
$\langle X,Y\rangle=0$, and
$\langle X,Y'\rangle=\langle X',Y\rangle=0$ for all $t\in(\R,0)$,
namely, the following holds.
\begin{lemma}{\rm (\cite[Lemma 3.2]{sing2ruledhypersurface})}
	For any germs\/ $X,Y:(\R,0)\to(\R^4,0)$ that are linearly independent
	for all\/ $t\in(\R,0)$,
	there exists\/ $\tilde{X},\tilde{Y}:(\R,0)\to(\R^4,0)$ 
	satisfying\/
	$\langle X,Y\rangle_{\R}=\langle \tilde{X},\tilde{Y}\rangle_{\R}$,
	$|\tilde{X}|=|\tilde{Y}|=1$,
	$\langle \tilde{X},\tilde{Y}\rangle=0$, and\/ 
	$\langle \tilde{X},\tilde{Y}'\rangle=
	\langle \tilde{X}',\tilde{Y}\rangle=0$ for all\/ $t\in(\R,0)$,
	where\/ $'=d/dt$. 
\end{lemma}
Here, $\inner{~}{~}$ denotes the standard inner product in $\R^4$,
and
$\langle X_1,\ldots,X_k\rangle_{\R}$
means the subspace spanned by $X_1,\ldots,X_k$.
We say that the director curves $X$ and $Y$ are
{\it constrictively adapted\/} if they
satisfy the conditions
$|X|=|Y|=1$, 
$\inner{X}{Y}=0$
and
$\inner{X}{Y'}=0$ for all\/ $t\in(\R,0)$.
In this paper, we assume that 
the director curves of 
all two-ruled hypersurfaces are 
taken to be constrictively adapted.
A map $P:\R\to \gr$ is {\it non-degenerate at\/} $t\in\R$ if 
for a basis $\{X(t),Y(t)\}$ of $P(t)$ satisfies
$X(t)$, $Y(t)$, $X'(t)$, $Y'(t)$ are linearly independent.
The non-degeneracy does not depend on the choice of 
basis and parameter.
A two-ruled hypersurface-germ $f(t,s,r) = \gamma(t) + s X(t) + r Y(t)$
is said to be \emph{non-degenerate\/} if
$P=\langle X,Y\rangle_{\R}
:(\R,0)\to\gr$ is non-degenerate for all $t\in(\R,0)$.
A two-ruled hypersurface $f(t,s,r) = \gamma(t) + s X(t) + r Y(t)$
is a {\it cylinder}\/ if
$P=\langle X,Y\rangle_{\R}$ is a constant plane.
We remark that $f$ is a cylinder if and only if
$
\dim\langle X,Y,X',Y'\rangle_{\R}=2
$
for all $t\in(\R,0)$.
For a non-degenerate two-ruled hypersurface 
$f(t,s,r) = \gamma(t) + s X(t) + r Y(t)$,
there exist functions 
$B(t)=(b_1(t),b_2(t),b_3(t),b_4(t))$ such that
\begin{equation}\label{eq:nondegb}
	\gamma'=b_1X+b_2Y+b_3X'+b_4Y'
\end{equation}
holds for all $t\in(\R,0)$.
Different from the case of ruled surfaces in $\R^3$, 
in four dimensions, 
there is a gap between non-degenerate two-ruled hypersurfaces and cylinders.
We introduce another notion of non-degeneracy.
A map $P:\R\to\gr$ is said to be {\it pseudo-non-degenerate at\/} $t\in\R$ if 
a basis $\{X(t),Y(t)\}$ of $P(t)$ satisfies
$$
\dim\left\langle X(t), Y(t), X'(t), Y'(t) \right\rangle_{\R}=3.
$$
The pseudo-non-degeneracy does not depend on the choice of 
basis and parameter.
\begin{definition}
	A two-ruled hypersurface $f(t,s,r) = \gamma(t) + s X(t) + r Y(t)$
	is said to be \emph{pseudo-non-degenerate}\/ if
	$P=\langle X,Y\rangle_{\R}:(\R,0)\to\gr$ 
	is pseudo-non-degenerate for all $t\in(\R,0)$.
\end{definition}
For a pseudo-non-degenerate two-ruled hypersurface
$f(t,s,r) = \gamma(t) + s X(t) + r Y(t)$ with constrictively 
adapted director curves $X$ and $Y$,
either
$\dim\left\langle X, Y, X'\right\rangle_{\R}=3$ or
$\dim\left\langle X, Y, Y'\right\rangle_{\R}=3$ holds.
By interchanging $X$ and $Y$ if necessary, 
we may assume 
$$\dim\left\langle X,Y,X'\right\rangle_{\R}=3$$
holds for any $t\in(\R,0)$ for a pseudo-non-degenerate two-ruled hypersurface.
By a parameter change, we assume $|X'|=1$ for any $t\in(\R,0)$, 
where pseudo-non-degenerateness
and constrictively adaptedness do not change.
Under this assumption, there exist functions 
$a(t)$ and $B(t)=(b_1(t),b_2(t),b_3(t),b_4(t))$ such that
\begin{equation}\label{eq:nondegp}
	\begin{array}{rl}
		Y'=&aX'\\
		\gamma'=&b_1X+b_2Y+b_3X'+b_4Z\quad
		\left(Z=X\wedge Y\wedge X'\right)
	\end{array}
\end{equation}
holds for any $t\in(\R,0)$. 
Here, we define the triple exterior product of
$X_1,X_2,X_3\in\R^4$ by
$$
X_1\wedge X_2\wedge X_3=\det
\pmt{
	\e_1&\e_2&\e_3&\e_4\\
	X_{11}&X_{12}&X_{13}&X_{14}\\
	X_{21}&X_{22}&X_{23}&X_{24}\\
	X_{31}&X_{32}&X_{33}&X_{34}},
$$
where
$X_i=(X_{i1},X_{i2},X_{i3},X_{i4})$ $(i=1,2,3)$
and
$\e_1=(1,0,0,0)$,
$\e_2=(0,1,0,0)$,
$\e_3=(0,0,1,0)$,
$\e_4=(0,0,0,1)$.
Since $\{X,Y,X',Z\}$ forms an orthonormal frame field along $\gamma$,
we have the following Frenet-Serre type formulas:
\begin{equation}\label{frenetxy}
	\pmt{X\\Y\\X'\\Z}'
	=A
	\pmt{X\\Y\\X'\\Z},\quad
	A=\pmt{
		0 &0 &1&      0\\
		0 &0 &a&      0\\
		-1&-a&0&\delta\\
		0 &0 &-\delta&0},
\end{equation}
where $\delta=-\det(X,Y,X',X'')$.
Since $A\in \mathfrak{o}(4)$, there exists a
unique orthonormal frame $\{X(t),Y(t),X'(t),Z(t)\}$ from a given functions
$a(t),\delta(t)$
satisfying \eqref{frenetxy} up to isometries on $\R^4$.
Furthermore, 
there exists unique $\gamma$ satisfying 
\eqref{eq:nondegp} from a given 
orthonormal frame $\{X(t),Y(t),X'(t),Z(t)\}$ and
$B(t)=(b_1(t),b_2(t),b_3(t),b_4(t))$
up to translations in $\R^4$.

In this paper, for a pseudo-non-degenerate two-ruled hypersurface $f=\gamma+sX+rY$,
we assume $X$ and $Y$ are constrictively adapted and $|X'|=1$.
We discuss frontality of non-degenerate and pseudo-non-degenerate
two-ruled hypersurfaces.
A map $f:(\R^3,0)\to(\R^4,0)$ is called a {\it frontal\/} 
if there exists a map $\nu: (\R^3,0) \to \R^4$ 
such that $|\nu| = 1$
and for any point $p \in (\R^3,0)$, 
and any vector $X \in T_p \R^3$, 
it holds that $\inner{df_p(X)}{\nu(p)}= 0$.
The map $\nu$ 
is called the {\it unit normal vector\/} of $f$.
A frontal $f$ is called a {\it front\/} 
if $(f, \nu)$ is an immersion.
We have the following proposition.
\begin{proposition}\label{prop:frontal}
	Let\/ $f(t,s,r) = \gamma(t) + s X(t) + r Y(t):(\R^3,(0,s,r))\to(\R^4,0)$
	be a two-ruled hypersurface-germ.
	$(1)$ If\/ $f$ is non-degenerate, then\/ $f$ never be a frontal at
	a singular point.
	$(2)$
	If\/ $f$ is pseudo-non-degenerate,
	then\/ $f$ is a frontal at a singular point
	if and only if\/ $b_4=0$ for all $t\in(\R,0)$.
\end{proposition}
\begin{proof}
	(1)
	Let $f$ be non-degenerate, and let
	$(b_1,\ldots,b_4)$ be functions defined by \eqref{eq:nondegb}.
	Then we see
	\begin{align*}
		f_t=&b_1X+b_2Y+(b_3+s)X'+(b_4+r)Y',\\
		f_s=&X,\\
		f_r=&Y,\\
		f_t\wedge f_s\wedge f_r=&
		(b_3+s)X\wedge Y\wedge X'+(b_4+r)X\wedge Y\wedge Y'.
	\end{align*}
	Then the limit of the direction of $f_t\wedge f_s\wedge f_r$
	exists if and only if there exist functions $\alpha(t,s,r),\beta(t,s,r)$
	such that
	$\alpha(t,s,r)(b_3+s)+\beta(t,s,r)(b_4+r)=0$
	and $(\alpha(t,s,r)$, $\beta(t,s,r))\ne(0,0)$ hold.
	However, $s$ and $r$ are independent variables,
	this never occurs.
	Thus (1) is proved.
	(2)
	Let $f$ be pseudo-non-degenerate, and let
	$a$ and $(b_1,\ldots,b_4)$ be functions defined by \eqref{eq:nondegp}.
	Then we see
	\begin{equation}\label{eq:difff}
		\begin{array}{rl}
			f_t=&b_1X+b_2Y+(b_3+s+ra)X'+b_4Z,\\
			f_s=&X,\\
			f_r=&Y,\\
			f_t\wedge f_s\wedge f_r=&
			(b_3+s+ra)X\wedge Y\wedge X'+b_4X\wedge Y\wedge Z.
		\end{array}
	\end{equation}
	At a singular point, $b_4=0$.
	If $b_4\equiv0$, then $X\wedge Y\wedge X'$ can be taken
	as a unit normal vector,
	$f$ is a frontal,
	where $\equiv$ stands for the equality holds identically.
	If $b_4$ is non constant near a singular point, 
	then the limit of the direction of $f_t\wedge f_s\wedge f_r$
	exists if and only if there exist functions $\alpha(t,s,r),\beta(t,s,r)$
	such that
	$\alpha(t,s,r)(b_3+s+ra)+\beta(t,s,r)b_4=0$
	and $(\alpha(t,s,r),\beta(t,s,r))\ne(0,0)$ hold.
	However, $s$ is an independent variable,
	this never occurs.
	Thus, (2) is proved.
\end{proof}
We call 
a pseudo-non-degenerate two-ruled hypersurface
with $b_4\equiv0$
a {\it pseudo-non-degenerate two-ruled frontal}.
For frontness of a pseudo-non-degenerate two-ruled frontal,
we have the following.
\begin{proposition}
	Let\/ $f(t,s,r) = \gamma(t) + s X(t) + r Y(t)$ be a pseudo-non-degenerate
	two-ruled frontal.
	Then\/ $f$ is a front at a singular point
	if and only if\/ $X,Y,X',X''$ are linearly independent at that point.
\end{proposition}
\begin{proof}
	By the proof of Proposition \ref{prop:frontal},
	$f$ is a frontal with a unit normal vector 
	$\nu=X\wedge Y\wedge X'$.
	By \eqref{eq:difff},
	the kernel of $df$ at a singular point is
	generated by $\eta
	=\partial_t-b_1\partial_s-b_2\partial_r$.
	Thus $\eta(X\wedge Y\wedge X')=X\wedge Y\wedge X''$ holds.
	Thus $f$ is a front if and only if
	$X\wedge Y\wedge X'$ and 
	$X\wedge Y\wedge X''$
	are linearly independent.
	Consider the linear map
	$T:\R^4\to\R^4$ defined by $T(V)=X\wedge Y\wedge V$.
	Since $X$ and $Y$ are linearly independent, we have
	$\ker T=\langle X,Y\rangle_{\R}$.
	Thus the image of $T$ is a $2$-dimensional subspace of $\R^4$.
	Now observe that
	$
	T(X')=X\wedge Y\wedge X'$ and
	$T(X'')=X\wedge Y\wedge X''$,
	the vectors $T(X')$ and $T(X'')$ are linearly independent in 
	$\operatorname{Im}T$
	if and only if their classes in the quotient space
	\[
	\R^4 / \ker T \;\cong\; \R^4 / \langle X,Y\rangle_{\R}
	\]
	are linearly independent.
	This is equivalent to that 
	$X,Y,X',X''$ are linearly independent,
	we have the assertion.
\end{proof}

\subsection{Striction surfaces and striction curves}
We discuss an analogy of striction curves for two-ruled hypersurfaces.
We begin with a geometric characterization of striction
curves for ruled surfaces.
Let $\gamma:(\R,0) \to (\R^n,0)$ be a curve-germ.
Let $X:(\R,0) \to \R^n$ be a vector such that 
$|X|=1$ holds.
A {\it ruled surface\/} is defined by
\begin{equation}\label{eq:ruled}
	g(t,s) = \gamma(t) + s X(t).
\end{equation}
A curve $\sigma(t)=\gamma(t)+s(t)X(t)$ is called a {\it striction curve\/}
if $\sigma'\cdot X'\equiv0$ holds.
The ruled surface $g$ is said to be {\it non-cylindrical\/} if
$X'\ne0$ for all $t\in(\R,0)$.
A curve $\sigma(t)=\gamma(t)+s(t)X(t)$ satisfying
$s(t)=-\inner{\gamma'(t)}{X'(t)}/\inner{X'(t)}{X'(t)}$
on a non-cylindrical ruled surface is a striction curve.
If $X'\equiv0$, then for any $s(t)$, the curve
$\sigma(t)=\gamma(t)+s(t)X(t)$ is a striction curve.
See \cite{hayashi} for the behavior of striction curves
in the case $X'(0)=0$ and $X'$ is not identically zero.
We give a characterization of striction curve 
by considering the distance between adjacent rulings.
We fix a parameter $t_0$ and assume $t_0=0$, and fix
a small number $\ep$.
We have
$$
\gamma(\ep)=\gamma(0)+\ep \gamma'(0)+\ep^2a(\ep),\quad
X(\ep)=X(0)+\ep X'(0)+\ep^2b(\ep)
$$
We define a function
\begin{align*}
	d(\ep,s_1,s_2)
	=&
	\big|\gamma(0) + s_1 X(0)-\gamma(\ep) - s_2 X(\ep)\big|^2/2\\
	=&
	\big|(s_1-s_2)X(0)-\ep(\gamma'(0)+s_2X'(0))-\ep^2(a(\ep)+s_2b(\ep))\big|^2/2.
\end{align*}
This measures the distance between 
each point on the rulings at $t=0$ and $t=\ep$.
The differential $d_{s_1}$ and $d_{s_2}$ are
\begin{align}
	d_{s_1}=&X(0)\cdot\Big((s_1-s_2)X(0)-\ep(\gamma'(0)+s_2X'(0))-\ep^2(a(\ep)+s_2b(\ep))\Big),
	\label{eq:ds1} \\
	d_{s_2}=&-X(0)\cdot\Big((s_1-s_2)X(0)-\ep(\gamma'(0)+s_2X'(0))-\ep^2(a(\ep)+s_2b(\ep))\Big)
	\label{eq:ds2} \\
	&\hspace{5mm}-\ep X'(0)\cdot\Big((s_1-s_2)X(0)-\ep(\gamma'(0)+s_2X'(0))-\ep^2(a(\ep)+s_2b(\ep))\Big)
	\nonumber \\
	&\hspace{5mm}-\ep^2b(\ep)\cdot\Big((s_1-s_2)X(0)-\ep(\gamma'(0)+s_2X'(0))-\ep^2(a(\ep)+s_2b(\ep))\Big).
	\nonumber
\end{align}
By \eqref{eq:ds1}, the equation $d_{s_1}=0$ is equivalent to
$$
s_1=s_2+\ep(\gamma'(0)\cdot X(0))+\ep^2(a(\ep)+s_2b(\ep))\cdot X(0).
$$
Substituting this into \eqref{eq:ds2},
we have
$$
d_{s_2}
=
\ep^2
\Big(\gamma'(0)\cdot X'(0)+s_2X'(0)\cdot X'(0)+\ep *\Big),
$$
where $*$ stands for a value that we do not use in the further
calculations.
Thus $(d_{s_1},d_{s_2})=(0,0)$ is equivalent to
$$
s_1=-b_3(0)+\ep *,\quad
s_2=-b_3(0)+\ep *,\quad
\left(
b_3(0)=\dfrac{\gamma'(0)\cdot X'(0)}{X'(0)\cdot X'(0)}\right).
$$
This implies that the striction curve is the locus of
points that minimize, in an infinitesimal sense,
the distance between adjacent rulings.
The definition of the striction curve for two-ruled hypersurfaces
in \cite{sing2ruledhypersurface} has this property.
A ruled surface \eqref{eq:ruled} with $|X|=1$
is a {\it cylinder\/} if $X'\equiv0$.
A non-cylindrical ruled surface \eqref{eq:ruled}
is a {\it cone\/} if $\sigma'\equiv0$,
namely, the image of the striction curve is a single point.
A non-cylindrical ruled surface \eqref{eq:ruled}
is a {\it tangent surface\/} if $\gamma'$ and $X$ is linearly dependent for all
$t\in (\R,0)$.

Let $f(t,s,r)=\gamma(t)+sX(t)+rY(t)$ be a two-ruled hypersurface,
and $X,Y$ are constrictively adapted.
A curve $\sigma(t)=\gamma(t)+s(t)X(t)+r(t)Y(t)$ is a striction curve
if $\sigma'\cdot X'=\sigma'\cdot Y'=0$ for any $t$ holds.
We define
\begin{align*}
	d(\ep,s_1,s_2,r_1,r_2)=&
	\Big|\gamma(0)+s_1X(0)+r_1Y(0)-(\gamma(\ep)+s_2X(\ep)+r_2Y(\ep))\Big|^2/2\\
	=&
	\Big|(s_1-s_2)X(0)+(r_1-r_2)Y(0)\\
	&\hspace{5mm}
	-\ep(\gamma'(0)+s_2X'(0)+r_2Y'(0))+\ep^2*\Big|^2/2.
\end{align*}
Then we see
\begin{align*}
	d_{s_1}=&X(0)\cdot\Big(
	(s_1-s_2)X(0)+(r_1-r_2)Y(0)
	-\ep(\gamma'(0)+s_2X'(0)+r_2Y'(0))+\ep^2*\Big),\\
	d_{r_1}=&Y(0)\cdot\Big(
	(s_1-s_2)X(0)+(r_1-r_2)Y(0)
	-\ep(\gamma'(0)+s_2X'(0)+r_2Y'(0))+\ep^2*\Big).
\end{align*}
Thus $(d_{s_1},d_{r_1},d_{s_2},d_{r_2})=(0,0)$
is equivalent to $s_1=s_2+\ep*$, $r_1=r_2+\ep*$ and
\begin{align*}
	X'(0)\cdot X'(0)s_2+X'(0)\cdot Y'(0)r_2=&-\gamma'(0)\cdot X'(0)+\ep*\\
	X'(0)\cdot Y'(0)s_2+Y'(0)\cdot Y'(0)r_2=&-\gamma'(0)\cdot Y'(0)+\ep*.
\end{align*}
Thus, the striction curve is also the locus of points that
minimize, in an infinitesimal sense,
the distance between adjacent rulings.
If $f$ is non-degenerate, then the matrix
$$
A=\pmt{X'\cdot X'&X'\cdot Y'\\
	X'\cdot Y'&Y'\cdot Y'}
$$
is regular for all $t$, and there exists a unique striction curve.
If $f$ is pseudo-non-degenerate, then 
$\sigma'\cdot X'=\sigma'\cdot Y'=0$ and $\sigma'\cdot X'=0$ are
equivalent. Further, 
both the matrices $A$ and
$$
\pmt{X'\cdot X'&X'\cdot Y'&-\gamma'\cdot X'\\
	X'\cdot Y'&Y'\cdot Y'&-\gamma'\cdot Y'}
$$
are of rank one for all $t$,
there exists a surface $\tilde\sigma(t,k)=\gamma(t)+s(t,k)X(t)+r(t,k)Y(t)$
such that for any curve $\sigma(t)=\gamma(t)+s(t)X(t)+r(t)Y(t)$
that satisfies $\image \sigma\subset \image \tilde\sigma$ 
is a striction curve.
We call this surface $\tilde\sigma$ a {\it striction surface\/}
for a pseudo-non-degenerate two-ruled hypersurface.
We set
\begin{equation}\label{eq:str1}
	s(t,r)=-a(t)r-b_3(t)\quad
	\left(b_3(t)=\dfrac{\gamma'(t)\cdot X'(t)}{X'(t)\cdot X'(t)}\right).
\end{equation}
Then $\tilde\sigma(t,k)=\gamma(t)+s(t,k)X(t)+kY(t)$ is a striction surface.
In fact, 
$\tilde\sigma(t,k(t))'\cdot X'(t)=(\gamma'+(-ak-b_3)X'+kY')\cdot X'=0$
holds for any function $k(t)$.
Substituting $s(t,r)$ into the formula of $f$,
we have an explicit parameterization
\begin{equation}\label{eq:tctx}
	\sigma_1(t,r)=\tilde \gamma(t)+r\tilde X(t)\quad
	\left(
	\tilde \gamma(t)=
	\gamma(t)-b_3(t)X(t),\ 
	\tilde X(t)=
	\dfrac{-a(t)X(t)+Y(t)}{|-a(t)X(t)+Y(t)|}
	\right)
\end{equation}
for the striction surface.
If $a'\ne0$, then this is a non-cylindrical ruled surface,
and the striction curve of $\sigma_1$ is
\begin{equation}\label{eq:str2}
	\sigma_2(t)=\tilde \gamma(t)+r(t)\tilde X(t),\quad
	r(t)=-\dfrac{\tilde \gamma'(t)\cdot \tilde X'(t)}{\tilde X'(t)\cdot \tilde X'(t)}.
\end{equation}
We call $\sigma_2$ the {\it second striction curve}.
By a direct calculation, we have
\begin{equation}
	\label{eq:c'X'}
	\begin{array}{rl}
		\tilde{\gamma}'&=(b_1-b_3')X+b_2Y,\quad
		\tilde{X}'=-a'(a^2+1)^{-1/2}(X+aY),\\
		\inner{\tilde{X}'}{\tilde{X}'}
		&=
		\dfrac{(a')^2}{(a^2+1)^2},\quad
		\inner{\tilde{\gamma}'}{\tilde{X}'}
		=
		\dfrac{-a'}{(a^2+1)^{3/2}}\omega,
	\end{array}
\end{equation}
where we set
\begin{equation}\label{eq:omega}
	\omega(t)=a(t)b_2(t)+b_1(t)-b_3'(t).
\end{equation}
Thus we obtain a parametrization of second striction curve
\begin{equation}\label{eq:sigma2}
	\sigma_2(t)=\tilde \gamma(t)+
	\dfrac{\omega(t)(a(t)^2+1)^{1/2}}{a'(t)}
	\tilde X(t).
\end{equation}
We consider the conditions under which 
striction surface is special.
A pseudo-non-degenerate two-ruled hypersurface 
is said to be {\it cylinder type\/} 
(respectively, {\it cone type\/}, {\it tangent developable type\/}) 
if whose striction surface is a cylinder (respectively, a cone,
a tangent developable surface).
By the above arguments,
the image of striction surface of a 
pseudo-non-degenerate two-ruled hypersurface 
never be a single point.
Moreover, the image of it cannot be 
a curve unless the original hypersurface is of cylinder, 
cone type or tangent developable type.
\begin{theorem}
	A pseudo-non-degenerate two-ruled hypersurface 
	$f(t,s,r) = \gamma(t) + s X(t) + r Y(t)$
	constructed from $a,\delta,B$ is
	of cylinder type
	if and only if $a'\equiv0$,\/
	and\/ $f$ is of cone type if and only if
	$$
	a'\ne0
	\quad\text{and}\quad
	a''(t)\omega(t)-a'(t)\Big(\omega'(t)+a'(t)b_2(t)\Big)\equiv0.
	$$
Moreover, $f$ is tangent developable type if and only if
$\omega\equiv b_4\equiv0$.
\end{theorem}
\begin{proof}
By \eqref{eq:c'X'}, 
the striction surface is a cylinder if $\tilde X'=0$ holds for any $t$.
The striction surface is a cone if $\sigma_2'=0$ holds for any $t$.
By \eqref{eq:sigma2}, we have
\begin{align*}
\sigma_2'(t)=\dfrac{a''(t)\omega-a'(t)\big(\omega'+a'(t)b_2(t)\big)}{a'(t)^2}\Big(a(t)X(t)-Y(t)\Big).
\end{align*}
Then we obtain the assertion for the cone type.
Differentiating \eqref{eq:tctx},
we see $\tilde\gamma'=(b_1-b_3')X_1+b_2Y+b_4Z$.
Since the director curve of $\sigma_1$ is proportional to $-aX+Y$,
the assertion for the tangent developable type follows.
\end{proof}
\subsection{Singularities and their criteria}\label{sec:sing}
\begin{definition}
	Two map germs $f,g:(\R^3,0)\to(\R^4,0)$ 
	are said to be {\it $\A$-equivalent} 
	if there exist a diffeomorphism $\phi:(\R^3,0)\to(\R^3,0)$ of the domain 
	and a diffeomorphism $\Phi:(\R^4,0)\to(\R^4,0)$ of the codomain 
	such that
	$$
	\Phi\circ f\circ \phi^{-1}=g.
	$$
\end{definition}
A map-germ $f$ is called a 
{\it Whitney umbrella\/ $\times$ interval\/}
if it is $\A$-equivalent to 
$(t,s,r)\mapsto(t,s,r^2,tr)$ at the origin.
A map-germ $f$ is called a {\it cuspidal edge\/} 
if it is $\A$-equivalent to $(t,s,r)\mapsto(t,s,r^2,r^3)$ at the origin.
A map-germ $f$ is called a {\it swallowtail\/} 
if it is $\A$-equivalent to $(t,s,r)\mapsto(t,s,3r^4+2tr^2,4r^3+tr)$ at the origin.
A map-germ $f$ is called a {\it cuspidal butterfly\/}
if it is $\A$-equivalent to 
$(t,s,r)\mapsto(t,s,-5r^4-4tr^2-2sr,4r^5+2tr^3+sr^2)$ at the origin.
A map-germ $f$ is called a {\it cuspidal cross cap\/} $\times$ {\it interval\/}
if it is $\A$-equivalent to 
$(t,s,r)\mapsto(t,s,r^2,sr^3)$ at the origin.
A Whitney umbrella\/ $\times$ interval is not a frontal,
a cuspidal cross cap\/ $\times$ interval is a frontal but not a front
at the singular points.
A cuspidal edge, swallowtail and cuspidal butterfly are fronts at the singular points.

There are useful methods to determine 
whether these singularities are of the types mentioned above.
\begin{lemma}
	Let\/ $f:(\R^3,0)\to(\R^4,0)$ be a map-germ, and\/
	$(t,s,r)$ be a coordinate system with\/ $f_t(0,0,0)=0$. 
	Then\/ $f$ is a Whitney umbrella\/ $\times$ interval if and only if
	$$
	\det(f_s,f_r,f_{rt},f_{tt})\neq 0
	\quad\text{or}\quad
	\det(f_s,f_r,f_{st},f_{tt})\neq 0
	$$
	at\/ $(t,s,r)=(0,0,0)$ hold.
\end{lemma}
Although this fact is well known, a proof can be found in 
\cite[Theorem 2.6]{morin}.

Let $f:(\R^3,0)\to(\R^4,0)$ be a frontal,
and let $\nu$ be a unit normal vector of $f$.
A function $\lambda$ 
is called a \emph{identifier of singularities\/}
if $\lambda$ is a non-zero scalar multiple of
$$
\det(f_t,f_s,f_r,\nu).
$$
If $\lambda$ is an identifier of singularities, 
$\lambda^{-1}(0)$ is the set of singular points $S(f)$.
The singular point $0$ of $f$ is said to be {\it non-degenerate\/}
if $d\lambda_0\ne0$.
If $0$ is non-degenerate, then $S(f)$ is a manifold near $0$.
We assume $\rank df_0=2$, then
there exists a vector field $\eta$ 
such that for any $p \in S(f)$, it holds that
$\ker df_p := \left\langle \eta_p \right\rangle_{\R}$.
This $\eta$ is called a {\it null vector field}.
We set
$$
S_2(f)=
\{p\in S(f)\,|\,\eta_p\in T_pS(f)\}
=
\{p\,|\,\lambda(p)=\eta\lambda(p)=0\}.
$$
We call $S_2(f)$ the {\it second singular set}.

Criteria for cuspidal edges, swallowtails and 
cuspidal butterflies are given through 
the identifier of singularities and the null vector field.

\begin{lemma}\label{lem:ak}
	{\rm (\cite[Corollary 2.5]{ak})}
	Let\/ $f:(\R^3,0)\to(\R^4,0)$ be a front with\/ $\rank df_0=2$. 
	Let\/ $\lambda$ be an identifier of singularities 
	and\/ $\eta$ be a null vector field.
	Then\/ $f$ is a cuspidal edge if and only if\/ 
	$
	\eta\lambda(0)\ne0
	$
	hold, $f$ is a swallowtail if and only if
	$$
	\rank d(\lambda,\eta\lambda)=2,\quad
	\eta\lambda(0)=0,\quad
	\eta\eta\lambda(0)\ne0
	$$
	hold, and\/ $f$ is a cuspidal butterfly if and only if it holds that
	$$
	\rank d(\lambda,\eta\lambda,\eta\eta\lambda)=3,\quad
	\eta\lambda(0)=0,\quad
	\eta\eta\lambda(0)=0,\quad
	\eta\eta\eta\lambda(0)\neq0.
	$$
\end{lemma}
For a non-front singularity, we have the following criteria.
Let $f:(\R^3,0)\to(\R^4,0)$ be a frontal,
and let $\nu$ be a normal vector of $f$, not necessary unit.
Let $\lambda$ be an identifier of  singularities,
and let $\eta$ be a null vector field.
We assume $\eta\lambda(0)\ne0$.
Then by this assumption, one can take a parametrization
$\Sigma(t,r)$ of $S(f)$.
We set $\hat f(t,r)=f(\Sigma(t,r))$.
We define a function $\psi(t,r)$ by
\begin{equation}\label{eq:defpsi1}
	\psi(t,r)=\det\Big(\hat f_t,\hat f_r,\nu(\Sigma(t,r)),(\eta\nu)(\Sigma(t,r))\Big).
\end{equation}
We have the following theorem.
\begin{theorem}\label{lem:ccri}
	Let\/ $f:(\R^3,0)\to(\R^4,0)$ be a frontal satisfying\/
	$\eta\lambda(0)\ne0$, where\/ 
	$\lambda$ is an identifier of singularities
	and\/ $\eta$ is a null vector field.
	Let\/ $\Sigma(u,v)$ be a parametrization of\/ $S(f)$ and
	we set\/ $\psi$ as in\/ \eqref{eq:defpsi1}.
	Then\/ $f$ is a cuspidal cross cap\/ $\times$ interval if and only if\/ 
	$\psi=0$ and\/ $d\psi_0\ne0$.
\end{theorem}
\begin{proof}
	We first see the condition does not depend on the choice
	of identifier of singularities,
	normal vector,
	parametrization of $S(f)$,
	null vector field and 
	coordinate system
	on the target space.
	It is clear that the independence of the choice of other than
	coordinate system
	on the target space,
	since replacing them only multiplies $\psi$ 
	by a non-vanishing function.
	For the choice of coordinate system on the target space,
	it is enough to show the following lemma.
	\begin{lemma}
		Under the above setting, the function\/ $\psi$
		is multiplied by a non-zero function 
		when the coordinates on the target space are replaced
		by a diffeomorphism.
	\end{lemma}
	\begin{proof}
		We take a diffeomorphism $\Phi$ on the target space.
		We set $\tilde f=\Phi\circ f$, and $\tilde \nu$ a normal vector of $\tilde f$.
		We set 
		$\tilde\psi=\det(\tilde f_t,\tilde f_r,\tilde\nu,\eta\tilde\nu)(\Sigma(t,r))$.We show $\tilde\psi$ is a non-zero functional
		multiplication of $\psi$.
		The differential $d\Phi$ can be regarded as a $GL(4,\R)$-valued
		map $W:(\R^3,0)\to GL(4,\R)$.
		Then a normal vector of $\tilde f$ is $\tilde \nu={}^tW^{-1}\nu$.
		Since $\eta$ is a null vector field, it holds that
		$$\tilde\psi=\det(Wf_t,Wf_r,{}^tW^{-1}\nu,{}^tW^{-1}\eta\nu)(\Sigma(t,r)).$$
		
		Since the assertion does not depend on the choice of
		parametrization of $S(f)$ and choice of $\nu$,
		we assume 
		$\hat f_t$, 
		$\hat f_r$ is an orthonormal frame on $S(f)$.
		Then 
		$\hat f_t$, 
		$\hat f_r$, $\nu$  is an orthonormal frame on $S(f)$,
		and $\eta\nu$ is perpendicular to these vectors.
		Taking an orthonormal frame of $T_p\R^4$ on $p\in S(f)$
		$\{e_1,\ldots,e_4\}$
		satisfying
		$\hat f_t=e_1$,
		$\hat f_r=e_2$,
		$\nu=e_3$,
		$\eta\nu=ke_4$.
		Then $\psi$ represented by this frame is $\psi=k$.
		By a direct calculation,
		we see
		$$
		\tilde\psi
		=
		\big(\inner{We_3}{We_3}\inner{We_4}{We_4}-\inner{We_3}{We_4}^2\big)
		\det W \psi,
		$$
		and this shows the assertion.
	\end{proof}
	We back to the proof of Theorem \ref{lem:ccri}.
	By the condition, we see $\rank df_0=2$, by a coordinate change on the source space,
 one can write
	$f$ as $(t,r,f_3(t,r,w),f_4(t,r,w))$.
	Taking $S(f)=\{w=0\}$ and null vector field is $\partial_w$
	by a coordinate change on the source space again,
	with a coordinate change on the target space,
	we can write $f$ as $(t,r,w^2f_3(t,r,w),w^2f_4(t,r,w))$.
	By the condition $\eta\lambda\ne0$, we have $(f_3,f_4)\ne(0,0)$ 
	at $0$,
	we may assume $f_3\ne0$.
	Moreover, we may assume $f_3>0$.
	We set $w=\sqrt{f_3}$.
	Then 
	one can write
	$f$ as $(t,r,w^2,w^2f_4(t,r,w))$.
	Since it is written as
	$(t,r,w^2,w^2(f_{41}(t,r,w^2)+wf_{42}(t,r,w^2))$,
	one can write
	$f$ as $(t,r,w^2,w^3f_4(t,r,w^2))$.
	Since the condition does not depend on the choice of
	adapted triple of vector fields and
	coordinate system on the target space,
	we set 
	$\xi_1=\partial_t$,
	$\xi_2=\partial_r$,
	$\eta=\partial_w$.
	Then $\psi$ is a non-zero multiple of $f_4$.
	Thus $((f_4)_t,(f_4)_r)\ne(0,0)$ holds.
	So we may assume $(f_4)_r\ne0$.
	Setting $\tilde r=f_4(t,r,w^2)$, we see $r$ can be written
	as
	$r=f_5(t,\tilde r,w^2)$.
	Thus $f$ is $\A$-equivalent to
	$(t,f_5(t,r,w^2),w^2,rw^3)$.
	Let us set 
	$f_0=(t,r,w^2,rw^3)$, and
	$\Phi(X_1,\ldots,X_4)=(X_1,f_5(X_1,X_2,X_3),X_3,X_4)$.
	Then $f=\Phi\circ f_0$ holds, and we see the assertion.
\end{proof}
Let $0$ be a non-degenerate singular point of a frontal $f$.
A triple of vector fields $(\xi_1,\xi_2,\eta)$ is said to be
{\it adapted\/} if $\xi_1,\xi_2$ are linearly independent and
tangent to $S(f)$, and $\eta$ is a null vector field.
We set 
\begin{equation}\label{eq:defpsi}
	\bar\psi=\det(\xi_1f,\ \xi_2f,\ \nu,\ \eta\nu),
\end{equation}
where $\zeta h$ stands for the directional derivative of a
function $h$ by a vector field $\zeta$.
Since the direction of $\eta$ is unique on the set of singular points,
we obtain the following corollary:
\begin{corollary}\label{cor:ccri}
	Let\/ $f:(\R^3,0)\to(\R^4,0)$ be a frontal with\/ $\eta\lambda(0)\ne0$.
	Let\/ $\lambda$ be an identifier of singularities
	and,\/ $\eta$ be a null vector field.
	We assume\/ $\eta\lambda\ne0$. 
	Let\/ $(\xi_1,\xi_2,\eta)$ be an adapted triple of vector fields.
	Then\/ $f$ is a cuspidal cross cap\/ $\times$ interval if and only if\/ 
	$\bar\psi=0$ and\/
	$(\xi_1\bar\psi,\xi_2\bar\psi)\ne(0,0)$ at\/ $0$,
	where\/ $(\xi_1,\xi_2,\eta)$ is an adapted triple of vector fields,
	and\/ $\bar\psi$ is the function defined in\/ \eqref{eq:defpsi}.
\end{corollary}

See \cite{fsuy} a criteria for cuspidal cross cap, and
\cite{ss}
a parameter-free version of the singular curve.

\subsection{Singularities of pseudo-non-degenerate two-ruled hypersurface}
In this subsection, we describe the singularities of 
pseudo-non-degenerate two-ruled hypersurfaces using 
the functions $a(t),\delta(t),B(t)=(b_1(t),\ldots,b_4(t))$ defined in 
\eqref{eq:nondegp}.

Let $f(t,s,r)=\gamma(t)+sX(t)+rY(t)$ 
be a pseudo-non-degenerate two-ruled frontal.
\begin{lemma}
	The set of singular points of a
	a pseudo-non-degenerate two-ruled frontal\/ $f$ coincides with
	the image of the striction surface.
	The second singular set coincides with
	the image of the second striction curve.
\end{lemma}
\begin{proof}
	By \eqref{eq:difff}, we see
	the identifier of singularities can be taken as
	$\lambda(t,s,r)=b_3(t)+s+ra(t)$,
	and the null vector field can be taken as
	$\eta=\partial_t-b_1\partial_s-b_2\partial_r$. 
	The singular set is $S(f)=\{(t,s,r)\,|\,b_3(t)+s+ra(t)=0\}$ and 
	the second singular set is 
	$$
	S_2(f)=\{p\in S(f)\,|\, ra'-\omega=0\}
	=\{(t,s,r)\,|\,r=\omega/a',\ 
	s=-b_3-a\omega/a'\},$$
	where $\omega$ is defined in \eqref{eq:omega}.
	Thus $f(S_2(f))$ is parameterized by
	$$
	\gamma-b_3X+
	\dfrac{\omega}{a'}(-aX+Y)
	=
	\tilde \gamma+
	\dfrac{\omega(a^2+1)^{1/2}}{a'}\tilde X,
	$$
	where $\tilde \gamma$ and $\tilde X$ are defined in \eqref{eq:tctx}.
	Comparing this with \eqref{eq:sigma2},
	we have the assertion.
\end{proof}
\begin{theorem}\label{thm:singcond}
	Let\/ $f=\gamma+sX+rY:(\R^3,(0,s,r))\to(\R^4,0)$ be a pseudo-non-degenerate 
	two-ruled frontal.
	Let\/ $p=(0,s,r)$ be a singular point of\/ $f$, namely,
	$b_3(t)+s+ra(t)=0$ holds.
	We assume\/ $f$ at\/ $p$ is a front, namely, $\delta(0)\ne0$.
	Then\/ $f$ at\/ $p$ is a cuspidal edge if and only if 
	\begin{equation}\label{eq:ce}
		r a'- a b_2 - b_1 + b_3'\ne0
	\end{equation}
	at\/ $t=0$ holds. 
	The germ\/ $f$ at\/ $p$ is a swallowtail if and only if\/
	$(A)$ or\/ $(B)$ holds, where
	\begin{enumerate}
		\item[$(A)$]
		$a'=0$, $- a b_2 - b_1 + b_3'=0$ and\/ $q_0\ne0$ at\/ $t=0$, where
		\begin{equation}\label{eq:sw0}
			q_0=- a b_2' + r a''  - 2 b_2 a' - b_1' + b_3''.
		\end{equation}
		\item[$(B)$]
		$r=(a b_2 + b_1 - b_3')/a'$ and\/ 
		$q_1\ne0$ at\/ $t=0$, where
		\begin{equation}\label{eq:sw}
			q_1=(b_2 a'' - a' b_2')
			+ (b_1 - b_3') a''
			- 2 b_2 (a')^2 - a' b_1' + a' b_3''.
		\end{equation}
	\end{enumerate}
	The germ\/ $f$ at\/ $p$ is a cuspidal butterfly if and only if\/
	$a'\ne0$, $r=(a b_2 + b_1 - b_3')/a'$, 
	$q_1=0$ and\/ $q_2\ne0$ at $t=0$, where
	\begin{equation}\label{eq:bt}
		q_2=a( b_2 a''' -  a' b_2'')
		+ a'''(b_1 - b_3')
		- 3 b_2 a' a'' - 3 (a')^2 b_2' + a' (-b_1'' + b_3''').
	\end{equation}
\end{theorem}
\begin{proof}
	Since the identifier of singularities $\lambda$ is
	$\lambda=b_3+s+ra$, and the null vector field
	$\eta$ is $\eta=\partial_t-b_1\partial_s-b_2\partial_r$. 
	We can directly apply Lemma \ref{lem:ak},
	and obtain the assertion.
	In the case of the cuspidal butterfly, we solve the equation 
	obtained by setting the left-hand side of \eqref{eq:sw} to zero, 
	substitute the solution into the condition $\eta\eta\eta\lambda\ne0$, 
	and obtain \eqref{eq:bt}.
\end{proof}
The case that $f$ is not a frontal,
namely, the case of $b_4\ne0$, we have the following.
\begin{theorem}\label{thm:singcondwu}
	Let\/ $f:(\R^3,0)\to(\R^4,0)$ be a 
	pseudo-non-degenerate two-ruled hypersurface.
	Then\/ $f$ is a Whitney umbrella\/ $\times$ interval
	if and only if\/ $b_4\neq0$ and\/ $b_4'\neq0$.
\end{theorem}
\begin{proof}
	Let $f$ be a pseudo-non-degenerate two-ruled hypersurface.
	By \eqref{eq:difff}, we know that
	$\det(f_s,f_r,f_{st},f_{tt})=ab_4'$ and $\det(f_s,f_r,f_{rt},f_{tt})=b_4'$,
	when $f_t=0$. Thus we obtain the result.
\end{proof}

The case that $f$ is a frontal but not a front,
we have the following.
\begin{theorem}\label{thm:singcond2}
	Let\/ $f:(\R^3,0)\to(\R^4,0)$ be a pseudo-non-degenerate 
	two-ruled frontal. 
	Let\/ $p=(t,s,r)$ be a singular point of\/ $f$, namely,
	$b_3(t)+s+ra(t)=0$ holds.
	Then\/ $f$ is a cuspidal cross cap\/ $\times$ interval
	if and only if\/ $\delta=0$ and\/ $\delta'\ne0$.
\end{theorem}
\begin{proof}
	Since $S(f)=\{b_3+s+ra\}$, we take
	$$
	\xi_1=\partial_t-(b_3'+r a')\partial_s,\quad
	\xi_2=-a\partial_s+\partial_r
	$$
	and $\eta$ as the above.
	By a direct calculation, we have
	$\bar\psi=\delta$.
	This shows the assertion.
\end{proof}
We give an example of pseudo-non-degenerate two-ruled
frontal who has a cuspidal cross cap $\times$ interval.
\begin{example}
	Let us set
	$f(t,s,r)=\gamma(t)+sX(t)+rY(t)$,
	where
	\begin{align*}
		\gamma(t)
		&=\Bigl(
		0,\,
		\int (\cos t-t^4\sin t)\,dt-t^4\cos t,\,
		\int (\sin t+t^4\cos t)\,dt-t^4\sin t,\\
		&\qquad
		\int (\sin t+t+t^4(\cos t+1))\,dt-t^4(\sin t+t)
		\Bigr),\\
		X(t)&=\bigl(0,\cos t,\sin t,\sin t+t\bigr),\\
		Y(t)&=\bigl(1,0,0,0\bigr).
	\end{align*}
	Then $f$ is a pseudo-non-degenerate two-ruled frontal.
	In fact, a unit normal vector field along $f$ is 
	\[
	\nu(t)
	=\dfrac{(0,t\cos t-\sin t,\,1+\cos t+t\sin t,\,-1)}{
		((t\cos t-\sin t)^2+(1+\cos t+t\sin t)^2+1)^{1/2}}.
	\]
	Moreover, the identifier of singularities and a null vector field of $f$ 
	can be taken by
	\begin{align*}
		\lambda(t,s,r)=\det(f_t,f_s,f_r,\nu)=s\quad\text{and}\quad
		\eta=\partial_t+(4t^3-1)\partial_s.
	\end{align*}
	Therefore, the singular set is
	$
	S(f)=\{(t,s,r)\mid s=0\}.
	$
	Hence, $\eta$ is a null vector field transverse to
	$S(f)$, and the singularities of $f$ are non-degenerate.
	A direct computation yields
	\[
	\psi(t,r)=\det\bigl(X,Y,X',X''\bigr)=-t.
	\]
	Therefore, the singularity of $f$ at $(0,0,r)$
	is a cuspidal cross cap $\times$ interval.
\end{example}

\subsection{Generic singularities of pseudo-non-degenerate two-ruled frontals}
In this section, we show that the generic singularities of
pseudo-non-degenerate two-ruled hypersurfaces
are
cuspidal edges, swallowtails, cuspidal butterflies or
cuspidal cross caps $\times$ an interval.
See \cite[Section 5]{sing2ruledhypersurface}
for generic singularities of non-degenerate two-ruled hypersurfaces.
Let $M$ be a one-dimensional manifold without boundary.
We set $PSF=\{d=(a,\delta,b_1,b_2,b_3)\in C^\infty(M,\R^5)\}$
to be the space
of pseudo-non-degenerate two-ruled frontals
equipped with the Whitney $C^\infty$ topology,
since such a frontal is
determined by these data, see Section \ref{sec:basic}.
\begin{theorem}
	There exists a dense subset\/ $\OO\subset PSF$ such that
	for any\/ $d=$ $(a$, $\delta$, $b_1$, $b_2$, $b_3)\in \OO$,
	the pseudo-non-degenerate two-ruled frontal\/ $f$
	determined by\/ $d$ 
	at any\/ $(t,s,r)\in M\times\R^2$ 
	is
	regular, cuspidal edge, swallowtail, cuspidal butterfly or
	cuspidal cross cap\/ $\times$ interval.
\end{theorem}
\begin{proof}
	Let $J^3(M,\R^5)$ be the three-jet bundle.
	The coordinate system on $J^3(M,\R^5)$ is 
	$x=(t,d,d',d'',d''')$.
	We define three algebraic subsets of $J^3(M,\R^5)$ by
	\begin{align*}
		Q_1&=\{j^5d(t)\in J^3(M,\R^5)\,|\,a'(t)=- a(t) b_2(t) - b_1(t) + b_3'(t)=0\},\\
		Q_2&=\{j^5d(t)\in J^3(M,\R^5)\,|\,q_1(t)=q_2(t)=0\},\\
		Q_3&=\{j^5d(t)\in J^3(M,\R^5)\,|\,\delta(t)=\delta'(t)=0\},
	\end{align*}
	where $q_1$ and $q_2$ are defined in \eqref{eq:sw}
	and \eqref{eq:bt}.
	We show $q_1$ and $q_2$ do not have any common factor.
	We assume $q_1$ and $q_2$ have a common factor $k$.
	Since $q_2$ is linear in $a'''$, 
	if $k$ contains $a'''$, then $k$ cannot divide $q_1$.
	Thus $k$ does not contain $a'''$.
	Since $k$ divides $q_2$, it holds that $k$ divides both
	$q_3=a b_2 + b_1 - b_3'$ and $q_4=-  aa' b_2''
	- 3 b_2 a' a'' - 3 (a')^2 b_2' + a' (-b_1'' + b_3''')$.
	Since $q_3$ is linear in $b_3'$,
	and $q_4$ does not contain $b_3'$,
	by the same argument, we obtain $k$ must be a non-zero real number.
	Thus $q_1$ and $q_2$ do not have any common factor.
	So, $Q_1,Q_2,Q_3$ are algebraic subset with codimension greater than 
	or equal to $2$.
	Thus the jet extension $j^5d:M\to J^3(M,\R^5)$ of
	a map $d:M\to \R^5$ is transverse to $Q_i$ $(i=1,2,3)$
	if and only if $j^5d(M)\cap Q_1\cap Q_2\cap Q_3=\emptyset$.
	By Theorems \ref{thm:singcond} and \ref{thm:singcond2},
	the condition
	$j^5d(M)\cap Q_1\cap Q_2\cap Q_3=\emptyset$
	is equivalent to
	the conclusion.
	By the Thom jet-transversality theorem,
	the set $\OO=\{d\in C^\infty(M,\R^5)\,|\,
	j^5d\text{ is transverse to }Q_1,Q_2,Q_3\}$
	is dense.
	This shows the assertion.
\end{proof}
One can show the generic singularities of 
pseudo-non-degenerate two-ruled hypersurfaces
is Whitney umbrella $\times$ interval by the same argument
using Theorem \ref{thm:singcondwu}.

\section{Two-ruled hypersurfaces constructed from height functions}
\subsection{A Frenet-type frame along a curve}
\label{sub:frame}
Let $\gamma: I \to \R^4$ be a curve in $\R^4$.
We assume there exist a function $l$ and a unit vector
$\e_1$ such that $\gamma'(t) = l \e_1$ holds.
We also assume there exist a function $\kappa_1$ and a
unit vector $\e_2$ such that
$
\e_1'=\kappa_1\e_2$ holds.
We also assume there exist a function $\kappa_4$ and a
unit vector $\e_3$ such that
$\e_2'+\kappa_1\e_1=\kappa_4\e_3$ holds.
We set
$\e_4 = \e_1\wedge \e_2\wedge \e_3$.
Then we have the following Frenet-Serret type formulas:
\begin{equation}\label{eq:frenet}
	\begin{pmatrix}
		\boldsymbol{e}_1\\
		\boldsymbol{e}_2\\
		\boldsymbol{e}_3\\
		\boldsymbol{e}_4\\
	\end{pmatrix}'
	= 
	\begin{pmatrix}
		0 & \kappa_1 & 0 & 0 \\
		-\kappa_1 & 0 & \kappa_4 & 0 \\
		0 & -\kappa_4 & 0 & \kappa_6 \\
		0 & 0 & -\kappa_6 & 0 \\
	\end{pmatrix}
	\begin{pmatrix}
		\boldsymbol{e}_1\\
		\boldsymbol{e}_2\\
		\boldsymbol{e}_3\\
		\boldsymbol{e}_4\\
	\end{pmatrix}.
\end{equation}
We call the frame $\{\e_1, \e_2, \e_3, \e_4\}$ 
constructed in the above manner, the {\it Frenet-type frame}
along $\gamma$.
\begin{definition}(\cite{ishikawa})
	A curve-germ $\gamma: (\R,0)\to (\R^4,0)$ is said to have
	{\it finite osculation type} {\rm (}respectively,
	{\it pseudo-finite osculation type\/}{\rm )}
	if it holds that $\rank G=4$ 
	(respectively, $3$), where
	$$
	G=\left(
	\gamma'(0),\gamma''(0),\ldots,\gamma^{(k)}(0),\ldots\right).
	$$
\end{definition}
A function, or a vector-valued function $h:(\R,0)\to(\R^k,0)$ $(k\geq1)$ is
said to have multiplicity $m$ if $h'(0)=\cdots=h^{(m-1)}(0)=0$ and
$h^{(m)}(0)\ne0$.
If there exists $m$ such that $h$ has the multiplicity $m$,
then $h$ is said to have {\it finite multiplicity}.
We assume $\gamma: (\R,0)\to (\R^4,0)$ has pseudo-finite osculation type.
Then there exist 
a function $l$ and a unit vector $\e_1$ such that $\gamma'(t) = l \e_1$ holds.
If $\e_1$ does not have finite multiplicity, then the rank of $G$ is one.
Thus we have a function $\kappa_1$ and a
unit vector $\e_2$ such that
$\e_1'=\kappa_1\e_2$ holds.
If $\e_2'+\kappa_1\e_1$ does not have finite multiplicity,
then the rank of $G$ is two.
Thus we have a function $\kappa_4$ and a
unit vector $\e_3$ such that
$\e_2'+\kappa_1\e_1=\kappa_4\e_3$ holds.
Therefore, 
for a pseudo-finite osculation type curve $\gamma$,
there exists a Frenet-type frame $\{\e_1,\ldots,\e_4\}$,
where $\e_4=\e_1\wedge\e_2\wedge\e_3$.
If $\gamma$ satisfies 
$\rank(\gamma',\gamma'',\gamma''')=3$,
then the above terminology is usual Frenet-Serret type argument.
See \cite[page 44]{fre2}, or \cite{fre1}.
By the assumption, each function $l,\kappa_1,\kappa_4$ has
finite multiplicity,
and if $\gamma$ has 
finite osculation type, then $\kappa_6$ also have finite multiplicity.
If one takes $\e_2$ from $\e_1^\perp$ and $\e_3$ from 
$\langle\e_1,\e_2\rangle_{\R}^\perp$,
then the above procedure runs well.
In that case, the matrix appearing in the Frenet-Serret type formula
\eqref{eq:frenet}
has the form
$$
\begin{pmatrix}
	0 &  \kappa_1 &  \kappa_2 &  \kappa_3 \\
	-\kappa_1 & 0 &  \kappa_4 &  \kappa_5 \\
	-\kappa_2 & -\kappa_4 & 0 &  \kappa_6 \\
	-\kappa_3 & -\kappa_5 & -\kappa_6 & 0 \\
\end{pmatrix}.
$$
For simplicity, we adopt the former framing.

\subsection{Developable hypersurfaces along a curve}
Let $\gamma:(\R,0)\to(\R^4,0)$ be a pseudo-finite osculation type curve,
and let $\{\e_1, \e_2, \e_3, \e_4\}$ be the Frenet-type frame.
For a unit vector field $\boldsymbol{v}$ along $\gamma$, 
we define the function $H_{\boldsymbol{v}}:I \times \R^4 \to \R$ 
by
$$
H_{\boldsymbol{v}}(t, x) = 
\boldsymbol{v}(t) \cdot \big(x-\gamma(t)\big).
$$
This is called the {\it height function with respect to $\boldsymbol{v}$}.
Furthermore, let us set $h_{\boldsymbol{v}}(t) = H_{\boldsymbol{v}}(t, 0)$. 
The function $H_{\boldsymbol{v}}$ can be interpreted as 
a 4-parameter family of 1-variable functions.
For each $t \in I$, the set
$$
\mathcal{H}_{\boldsymbol{v}}=\{x \in \R^4 \mid H_{\boldsymbol{v}}(t, x) = 0\}
$$
is a hyperplane orthogonal to $\boldsymbol{v}$.
Thus $\mathcal{H}_{\boldsymbol{v}}$ is a one-parameter family of planes.
Consider the envelope of this family of planes
$$
\mathcal{D}_{\boldsymbol{v}} 
=
\{x \in \R^4 \mid \text{ there exists }t \in I \text{ such that } 
H_{\boldsymbol{v}}(t, x) = H_{\boldsymbol{v}}'(t, x) = 0\},
$$
where $'=\partial/\partial t$.
Let $i=1,2,3,4$, we consider the cases for
$\boldsymbol{v}
=\{\e_i\}$ with the assumption $\boldsymbol{v}'\ne0$.
Then we get parametrizations of four envelopes 
$\mathcal{D}_{\e_i}$ constructed from $\gamma$.
Here and after, we omit the variable $(t)$ of functions if it does
not occur any confusion.

\begin{lemma}\label{lem:sai}
	We set four two-ruled hypersurfaces defined by
	\begin{align}\label{eq:si}
		S_1&(t,s,r)
		=\gamma+\left(\frac{l}{\kappa_1}\right)\e_2+s\e_3+r\e_4,\quad\text{where\/}\ \kappa_1\neq0;\nonumber\\
		S_2&(t,s,r)
		=\gamma+s\e_4+r\left(\frac{\kappa_4\e_1 + \kappa_1\e_3}{\sqrt{\kappa_4^2+\kappa_1^2}}\right),\quad\text{where\/}\ (\kappa_1,\kappa_4)\ne(0,0);\\
		S_3&(t,s,r)
		=\gamma+s\e_1+r\left(\frac{\kappa_6\e_2 + \kappa_4\e_4}{\sqrt{\kappa_6^2+\kappa_4^2}}\right),\quad\text{where\/}\ (\kappa_4,\kappa_6)\neq(0,0);\nonumber\\
		S_4&(t,s,r)
		=\gamma+s\e_1+r\e_2,\quad\text{where\/}\ \kappa_6\neq0.\nonumber
	\end{align}
	Then the image of\/ $S_{i}$
	coincides with the set\/ $\mathcal{D}_{\e_i}$,
	where\/ $i=1,2,3,4$.
	Moreover, 
	if\/ $\kappa_4\ne0$, then\/ $S_1$ and\/ $S_4$ are pseudo-non-degenerate,
	if\/ $(\kappa_4, \kappa_1\kappa_4')\ne(0,0)$, 
	then\/ $S_2$ is pseudo-non-degenerate and
	if\/ $(\kappa_4, \kappa_4'\kappa_6)\ne(0,0)$, 
	then\/ $S_3$ is pseudo-non-degenerate.
\end{lemma}

\begin{proof}
	We show the case $\boldsymbol{v}=\e_4$.
	By the condition $H_{\e_4}(t, x) = 0$, 
	it holds that there exist $c_1,c_2,c_3 \in \R$ such that
	$x-\gamma = c_1 \e_1 + c_2 \e_2 +c_3 \e_4$.
	By \eqref{eq:frenet} and $\gamma'=l \e_1$.
	Moreover, substituting $x-\gamma$ into the formula
	$H_{\e_4}'(t, x) = 0$, we get
	\begin{align*}
		H_{\e_4}'(t, x) 
		= \e_4' \cdot \big(x-\gamma\big)
		= -\kappa_6\e_3 \cdot (c_1 \e_1 + c_2 \e_2 +c_3 \e_3)
		= -\kappa_6 c_3.
	\end{align*}
	Thus $c_3=0$ under the condition $\kappa_6\ne0$.
	Thus we set
	$x-\gamma 
	= s\e_1+r\e_2$, where $s=c_1\in\R$ and $r=c_2\in\R$.
	Hence the image of 
	$$
	x(t,s,r)
	=\gamma + s\e_1+r\e_2
	$$ coincides with 
	$\mathcal{D}_{\e_4}$.
	By differentiating director curves in \eqref{eq:si}, we have
	$\e_1'=\kappa_1\e_2$, $\e_2'=-\kappa_1\e_1+\kappa_4\e_3$.
	Thus $S_4$ is pseudo-non-degenerate if and only if $\kappa_4\ne0$.
	We can show the case of $i=1,2,3$ by a similar calculation.
	The director curves of $S_1$ and $S_4$ and their differential are given by
	\begin{align*}
		X_1=\e_3,
		\quad
		X_1'=-\kappa_4\e_2+\kappa_6\e_4
		\quad&\text{and}\quad
		Y_1=\e_4,
		\quad
		Y_1'=-\kappa_6\e_3;\\
		X_4=\e_1,
		\quad
		X_4'=\kappa_1\e_2
		\quad&\text{and}\quad
		Y_4=\e_2,
		\quad
		Y_4'=-\kappa_1\e_1+\kappa_4\e_3.
	\end{align*}
	Therefore, we see that if $\kappa_4\ne0$, 
	then $\dim\{X_1,Y_1,X_1'\}=\dim\{X_4,Y_4,X_4'\}=3$.
	The director curves of $S_2$ and $S_3$ and their differential are given by
	\begin{align*}
		&X_2=\e_4,
		\quad\quad\quad\quad\quad\quad\
		X_2'=-\kappa_6\e_3;\\
		&Y_2=\frac{\kappa_4\e_1 + \kappa_1\e_3}{\sqrt{\kappa_1^2+\kappa_4^2}},
		\quad\quad
		Y_2'=\frac{\kappa_1\kappa_4'-\kappa_1'\kappa_4}{(\kappa_1^2+\kappa_4^2)^{3/2}}\Big(\kappa_1\e_1-\kappa_4\e_3\Big)+\frac{\kappa_1\kappa_6}{\sqrt{\kappa_1^2+\kappa_4^2}}\e_4.\\
		&X_3=\e_1,
		\quad\quad\quad\quad\quad\quad\
		X_3'=\kappa_1\e_2;\\
		&Y_3=\frac{\kappa_6\e_2 + \kappa_4\e_4}{\sqrt{\kappa_4^2+\kappa_6^2}},
		\quad\quad
		Y_3'=\frac{\kappa_1\kappa_6}{\sqrt{\kappa_4^2+\kappa_6^2}}\e_1+\frac{\kappa_4\kappa_6'-\kappa_4'\kappa_6}{(\kappa_4^2+\kappa_6^2)^{3/2}}\Big(\kappa_4\e_2-\kappa_6\e_4\Big).
	\end{align*}
	Therefore, we see that if $(\kappa_1\kappa_4',\kappa_4)\neq(0,0)$ and $(\kappa_4,\kappa_4'\kappa_6)\neq(0,0)$, then $S_2$ and $S_3$ are pseudo-non-degenerate.
\end{proof}
By the above calculations, we see that
a unit normal vector for $S_i$ are $\e_i$ $(i=1,2,3,4)$,
thus, they are frontal.
Since each $\e_i$ is a curve, each two-ruled frontal
$S_i$ is developable.

\section{Singularities of the two-ruled frontals $S_i$ ($i=1,2,3,4$)}
In this section, we describe the conditions 
under which the frontals $S_i$ obtained
in \eqref{eq:si} exhibit the singularities introduced in Section \ref{sec:sing} expressed in terms of the functions $l,\kappa_1,\kappa_4,\kappa_6$
of $\gamma$.

\subsection{Properties of the frontal $S_1$}
We assume $\kappa_1\kappa_4\ne0$, and 
set $\hat{\gamma}=\gamma+l\e_2/\kappa_1$.
Then $S_1$ can be written as $S_1(t,s,r)=\hat{\gamma}+s\e_3+r\e_4$.
We set $X_1=\e_3$ and $Y_1=\e_4$.
The partial derivatives are
\begin{align}\label{eq:s1'}
	(S_1)_t=\Big(\frac{l'\kappa_1-l\kappa_1'}{\kappa_1^2}-s\kappa_4\Big)\e_2
	+\Big(\frac{l\kappa_4}{\kappa_1}-r\kappa_6\Big)\e_3+s\kappa_6\e_4,\ 
	(S_1)_s=X_1,\ 
	(S_1)_r=Y_1.
\end{align}
We see
\begin{equation}\label{eq:l1eta1}
	\lambda_1=- s \kappa_4\kappa_1^2 + l'\kappa_1 - l\kappa_1',\quad
	\eta_1=\partial_t-(l \kappa_4/\kappa_1-r \kappa_6)\partial_s-s\kappa_6\partial_r
\end{equation}
can be taken as an identifier of singularities and 
a null vector field of $S_1$.
Since the director curves are not constrictively adapted,
we modify them so that they become constrictively adapted.
Let $\theta$ be a function satisfying $\theta'=\kappa_6$,
and define
\begin{align*}
	\hat{X_1}=\cos\theta X_1-\sin\theta Y_1,
	\quad\quad\quad
	\hat{Y_1}=\sin\theta X_1+\cos\theta Y_1.
\end{align*}
Then $\hat{X_1}$ and $\hat{Y_1}$ are constrictively adapted.
We choose $\theta$ satisfying $\cos\theta\ne0$.
Since $|(\hat{X_1})'|=|\cos\theta\kappa_4|$, we set
$$
\tilde t=\int^t |\cos\theta\kappa_4|dt=\tau(t).
$$
Then
$\hat{X_1}(\tau^{-1}(\tilde t))$,
$\hat{Y_1}(\tau^{-1}(\tilde t))$ and
$d\hat{X_1}(\tau^{-1}(\tilde t))/d\tilde t$
form an orthonormal set.
The functions $a,\delta$ and $B=(b_1,\ldots,b_4)$ defined in \eqref{eq:nondegp}
are given by
\begin{equation}\label{eq:aB1}
	a=\dfrac{\sin\theta}{\cos\theta},\ 
	\delta=\dfrac{\kappa_1}{|\cos\theta\kappa_4|},\ 
	B=
	\left(
	\dfrac{l\kappa_4\cos\theta}{\kappa_1|\cos\theta\kappa_4|},\ 
	\dfrac{l\kappa_4\sin\theta}{\kappa_1|\cos\theta\kappa_4|},\ 
	\dfrac{l\kappa_1'-l'\kappa_1}{\kappa_1^2(\kappa_4\cos\theta)^3},\ 0
	\right).
\end{equation}
Here, we note that 
$dt/d\tilde t=1/|(\hat{X_1})'|=1/|\cos\theta\kappa_4|$.
From these data, we can derive the striction surface and the striction
curve of $S_1$, as well as the correponding conditions of singularities,
using
\eqref{eq:str1}, \eqref{eq:str2} and Theorem \ref{thm:singcond},
respectively.
If we omit the condition $\kappa_4\ne0$, the frontal $S_1$ is defined.
\begin{theorem}
We assume $\kappa_1\kappa_4\ne0$.
The frontal \/ $S_1$ is of cylinder type if and only if\/ $\kappa_6\equiv0$.
The frontal \/ $S_1$ is of cone type
if and only if\/ $f_{1cone}\equiv0$, where
\begin{align*}
f_{1cone}=&l\Big\{
\kappa_1^3\kappa_4^3\bigl( \kappa_6 \kappa_4' - \kappa_4 \kappa_6'\bigr) 
+ \kappa_1^2\bigl(
-\kappa_1'\kappa_4^4 \kappa_6 
-2 \kappa_1' \kappa_4'^2\kappa_6 
-\kappa_1'\kappa_4 \kappa_4'  \kappa_6'
+\kappa_1' \kappa_4\kappa_4'' \kappa_6  \\
&\hspace{10mm}
-\kappa_1'\kappa_4^2 \kappa_6^3 
+2 \kappa_1''\kappa_4 \kappa_4' \kappa_6
+\kappa_1''\kappa_4^2 \kappa_6' 
-\kappa_1'''\kappa_4^2 \kappa_6 
\bigr) \\[4pt]
&\hspace{10mm}
+ \kappa_1\bigl(
-4 \kappa_1'^2\kappa_4 \kappa_4' \kappa_6 
-2 \kappa_1'^2\kappa_4^2 \kappa_6' 
+6 \kappa_1' \kappa_1''\kappa_4^2 \kappa_6 
\bigr) 
- 6 \kappa_1'^3\kappa_4^2 \kappa_6 
\Big\}\\
&+l'\kappa_1\Big\{
\kappa_1^2\bigl(
\kappa_4^4 \kappa_6
+ \kappa_4^2 \kappa_6^3
+ 2 \kappa_6 \kappa_4'^2
+ \kappa_4 \kappa_4' \kappa_6'
- \kappa_4 \kappa_4''\kappa_6 
\bigr) \\[4pt]
&\hspace{10mm}
+ \kappa_1\kappa_4\bigl(
4 \kappa_1'  \kappa_4' \kappa_6
+ 2 \kappa_1'\kappa_4 \kappa_6' 
- 3 \kappa_1''\kappa_4 \kappa_6 
\bigr) 
+ 6 \kappa_1'^2\kappa_4^2 \kappa_6 
\Big\}\\
&-l''\kappa_1^2 \kappa_4\Big\{
\kappa_1\bigl(2 \kappa_4'\kappa_6  + \kappa_4 \kappa_6'\bigl)
+
3 \kappa_1'\kappa_4 \kappa_6 
\Big\}
+l'''\kappa_1^3 \kappa_4^2 \kappa_6.
\end{align*}
\end{theorem}
\begin{proof}
The assertion follows by $a'=
2\kappa_6/(1 + \cos2\theta)$ 
and computing $a''\omega-a'(\omega'+a'b_2)$
	using \eqref{eq:aB1}.
\end{proof}

\begin{theorem}\label{thm:sings1}
	We assume $\kappa_1\kappa_4\ne0$.
	For the pseudo-non-degenerate two-ruled frontal\/ $S_1(t,s,r)$,
	let\/ $p=(t,s,r)$ be a singular point of\/ $S_1$, that is,
	$\lambda_1=0$ holds.
	Assume that\/ $S_1$ is a front at\/ $p$.
	Then\/ $S_1$ at\/ $p$ is a cuspidal edge if and only if\/ 
	$\W_{1c}\ne0$, where
	\begin{align}\label{eq:ce1}
		\W_{1c}=&\eta_1\lambda_1\nonumber\\
		=&r \kappa_1^3 \kappa_4^2 \kappa_6
		+\cos\theta\Big\{-2 \kappa_4 l (\kappa_1')^2
		+ \kappa_1 \bigl( 2 \kappa_4 \kappa_1' l' + l ( - \kappa_1' \kappa_4' + \kappa_4 \kappa_1'' ) \bigr)\\
		&\hspace{10mm}
		- \kappa_1^2 \bigl( \kappa_4^3 l - \kappa_4' l' + \kappa_4 l'' \bigr)\Big\}
		+\sin\theta\Big\{\kappa_1 \kappa_4 
		\kappa_6
		\bigl( l \kappa_1' - \kappa_1 l' \bigr) \Big\}\nonumber
	\end{align}
	holds. 
	The germ\/ $S_1$ at\/ $p$ is a swallowtail if and only if\/
	$\W_{1c}=0$ and\/ $\W_{1s}\ne0$, where
	\begin{equation}\label{eq:sw1}
		\W_{1s}=\eta_1\W_{1c}=
		\kappa_1(\W_{1c})_t
		- \kappa_4 l 
		\bigl( \sin\theta (\W_{1c})_{r} + \cos\theta (\W_{1c})_{s} \bigr).
	\end{equation}
	The germ\/ $S_1$ at\/ $p$ is a cuspidal butterfly if and only if\/
	$\W_{1c}=\W_{1s}=0$, $\kappa_6\ne0$ and\/ $\W_{1b}\ne0$, where
	\begin{equation}\label{eq:bt1}
		\W_{1b}=\eta_1\W_{1s}=
		\kappa_1(\W_{1s})_t
		- \kappa_4 l 
		\bigl( \sin\theta (\W_{1s})_{r} + \cos\theta (\W_{1s})_{s} \bigr).
	\end{equation}
	The germ\/ $S_1$ at\/ $p$ is a cuspidal cross cap\/ $\times$ interval 
	if and only if\/ $\kappa_1=0$
	and\/ $\kappa_1'\ne0$.
\end{theorem}
\begin{proof}
	Since the identifier of singularities, a null vector field
	and the function $\psi$ defined in \eqref{eq:defpsi1}
	can be taken as in \eqref{eq:l1eta1},
	we have the assertion.
\end{proof}
Since the second singular set coincides with the striction curve
of the striction surface,
$$
\big(\W_{1s}|_{\lambda=0,\,\W_{1c}=0}\big)(t)
$$
agrees with $f_{1cone}(t)$ up to a non-zero functional multiple.

\subsection{Properties of the frontal $S_2$}
We assume $(\kappa_1,\kappa_4)\ne(0,0)$, 
$(\kappa_4,\kappa_1\kappa_4')\ne(0,0)$.
Since the arguments for the case of $S_i$ $(i=2,3,4)$ are the same,
we just give the fundamental data $a,\delta,B$ and $\lambda,\eta$.
The partial differential with respect to $S_2$ is given by
\begin{align}
	(S_2)_t=&\e_1\Big(l+\frac{r \kappa_1 (\kappa_1 \kappa_4'-\kappa_4 \kappa_1')}{A_2^3}\Big)
	-\e_3\frac{s A_2^3 \kappa_6 +r\kappa_4(\kappa_1  \kappa_4'-\kappa_4 \kappa_1')}{A_2^3}
	+\e_4\frac{r \kappa_1 \kappa_6}{A_2},\nonumber\\
	(S_2)_s=&X_2=\e_4,\label{eq:s2'}\\
	(S_2)_r=&Y_2=\frac{\kappa_4\e_1 + \kappa_1\e_3}{A_2},\nonumber
\end{align}
where $A_2=(\kappa_1^2 + \kappa_4^2)^{1/2}$.
We see
\begin{align}
	\label{eq:l1}
	\lambda_2=&A_2(l\kappa_1
	+ s\kappa_4\kappa_6)
	+ r(-\kappa_4\kappa_1'
	+ \kappa_1\kappa_4'),\\
	\eta_2=&
	\partial_t
	-\frac{r\kappa_1\kappa_6}{A_2}\partial_s
	+\frac{
		sA_2^3\kappa_6
		+ r\kappa_4(-\kappa_4\kappa_1'+r\kappa_1\kappa_4')
	}{
		\kappa_1A_2^2}\partial_r
\end{align}
can be taken as an identifier of singularities and 
a null vector field of $S_2$.
Let $\theta$ be a function such that 
$\theta'=-\kappa_1\kappa_6/A_2$.
Setting
\begin{align*}
	\hat{X_2}=\cos\theta X_2-\sin\theta Y_2,
	\quad\quad\quad
	\hat{Y_2}=\sin\theta X_2+\cos\theta Y_2,
\end{align*}
we see $\hat{X_2}$ and $\hat{Y_2}$ are constrictively adapted
and 
$$
	|(\hat{X_2})'|^2=\dfrac{\Big(\kappa_4\kappa_6A_2\cos\theta
		+ (\kappa_4\kappa_1'- \kappa_1\kappa_4')\sin\theta\Big)^2}{(A_2)^{\frac{5}{2}}}.
	$$
We choose $\theta$ satisfying $|(\hat{X_2})'|\ne0$.
The functions $a,\delta$ and $B=(b_1,\ldots,b_4)$ defined in \eqref{eq:nondegp}
are given by
\begin{align}\nonumber
		a&=\frac{\kappa_4\kappa_6A_2\sin\theta+(\kappa_1\kappa_4'-\kappa_1'\kappa_4)\cos\theta}{\kappa_4\kappa_6A_2\cos\theta+(\kappa_1'\kappa_4-\kappa_1\kappa_4')\sin\theta},\\
		\delta&=
		\frac{A_2}{|(\hat{X_2})'|}\label{eq:aB2}\\
		B&=
		\left(
		-\dfrac{l\kappa_4\sin\theta}{A_2|(\hat{X_2})'|},\ 
		\dfrac{l\kappa_4\cos\theta}{A_2|(\hat{X_2})'|},\ 
		\dfrac{l\kappa_1A_2}{(\kappa_4\kappa_6A_2\cos\theta+(\kappa_1'\kappa_4-\kappa_1\kappa_4')\sin\theta)|(\hat{X_2})'|^2},\ 
		0
		\right).\nonumber
\end{align}
From these data, we can derive the striction surface and striction
curve of $S_2$, as well as
and the conditions for singularities by 
\eqref{eq:str1}, \eqref{eq:str2} and Theorem \ref{thm:singcond},
respectively.

Since the computations are similar and 
the explicit conditions themselves are 
not particularly illuminating, we omit the details.
\subsection{Properties of the frontal $S_3$}
We assume $(\kappa_4,\kappa_6)\ne(0,0)$, 
$(\kappa_4,\kappa_4'\kappa_6)\ne(0,0)$.
We have
\begin{align}
	(S_3)_t=&\e_1\Big(l-\frac{r\kappa_1\kappa_6}{A_3}\Big)
	+\e_2\Big(s\kappa_1+\frac{r\kappa_4(\kappa_4  \kappa_6'-\kappa_4'\kappa_6)}{A_3^3}\Big)+\e_4\frac{r \kappa_6(\kappa_4\kappa_6'-\kappa_4'\kappa_6)}{A_3^3},
	\nonumber\\
	(S_3)_s=&X_3=\e_1,\label{eq:s3'}\\
	(S_3)_r=&Y_3=\frac{\kappa_6\e_2 + \kappa_4\e_4}{A_3},\nonumber
\end{align}
where $A_3=\sqrt{\kappa_4^2+\kappa_6^2}$.
We see
\begin{align}\label{eq:l3}
	\lambda_3&=s\kappa_1\kappa_4A_3+ r(-\kappa_6\kappa_4'+ \kappa_4\kappa_6'),\\
	\eta_3&=
	\partial_t+
	\left(-l + \frac{r\kappa_1\kappa_6}{A_3}\right)\partial_s
	+
	\frac{
		- s\kappa_1A_3^3
		+ r\kappa_4\bigl(\kappa_6\kappa_4' - \kappa_4\kappa_6'\bigr)
	}{
		\kappa_6A_3^2
	}
	\partial_r
\end{align}
can be taken as an identifier of singularities and 
a null vector field of $S_1$.
Since the director curves are not constrictively adapted,
we modify them so that they become constrictively adapted.
Let $\theta$ be a function such that 
$\theta'=\kappa_1\kappa_6/A_3$.
Setting
\begin{align*}
	\hat{X_3}=\cos\theta X_3-\sin\theta Y_3,
	\quad\quad\quad
	\hat{Y_3}=\sin\theta X_3+\cos\theta Y_3,
\end{align*}
$\hat{X_3}$ and $\hat{Y_3}$ are constrictively adapted,
and
$$
	|(\hat{X_3})'|^2=
	\dfrac{\Big(\cos\theta\kappa_1\kappa_4A_3
		+ \sin\theta\bigl(\kappa_6\kappa_4' - \kappa_4\kappa_6'\bigr)\Big)^2}{(A_3)^4}.
$$
We choose $\theta$ satisfying $|(\hat{X_3})'|\ne0$.
The functions $a,\delta$ and $B=(b_1,\ldots,b_4)$ defined in \eqref{eq:nondegp}
are given by
\begin{align}\nonumber
		a&=\frac{\kappa_1\kappa_4A_3\sin\theta+(\kappa_4\kappa_6'-\kappa_4'\kappa_6)\cos\theta}{\kappa_1\kappa_4A_3\cos\theta+(\kappa_4'\kappa_6-\kappa_4\kappa_6')\sin\theta},\\
		\delta&=
		\frac{A_3}{|(\hat{X_3})'|}\label{eq:aB3}\\
		B&=
		\left(
		\dfrac{l\cos\theta}{|(\hat{X_3})'|},\ 
		\dfrac{l\sin\theta}{|(\hat{X_3})'|},\ 
		0,\ 
		0\right).\nonumber
\end{align}
From these data, we can derive the striction surface and striction 
curve of $S_3$,
as well as the conditions for singularities by 
\eqref{eq:str1}, \eqref{eq:str2} and Theorem \ref{thm:singcond},
respectively.
We omit the details.

\subsection{Properties of the frontal $S_4$}
We assume $\kappa_4\kappa_6\ne0$.
We have
\begin{align}\label{eq:s4'}
	(S_4)_t=\big(l-r\kappa_1\big)\e_1+s\kappa_1\e_2+r\kappa_4\e_3,\quad
	(S_4)_s=\e_1,\quad
	(S_4)_r=\e_2.
\end{align}
We see
\begin{equation}\label{eq:31eta3}
	\lambda_4=r\kappa_4,\quad
	\eta_4=\partial_t+(-l+r \kappa_1)\partial_s-s \kappa_1\partial_r
\end{equation}
can be taken as an identifier of singularities and 
a null vector field of $S_4$.
Since the director curves are not constrictively adapted,
we modify them so that they become constrictively adapted.
Let $\theta$ be a function such that $\theta'=\kappa_1$. 
Setting
\begin{align*}
	\hat{X_4}=\cos\theta X_4-\sin\theta Y_4,
	\quad\quad\quad
	\hat{Y_4}=\sin\theta X_4+\cos\theta Y_4,
\end{align*}
$\hat{X_4}$ and $\hat{Y_4}$ are constrictively adapted,
and
$|(\hat{X_4})'|^2=\kappa_4^2\sin^2\theta$.
We choose $\theta$ satisfying $\sin\theta\ne0$.
The functions $a,\delta$ and $B=(b_1,\ldots,b_4)$ defined in \eqref{eq:nondegp}
are given by
\begin{align}\label{eq:aB4}
	a=-\frac{\cos\theta}{\sin\theta},\ 
	\delta=\dfrac{\kappa_6}{|\kappa_4 \sin\theta|},\ 
	B=
	\left(\dfrac{l\cos\theta}{|\kappa_4 \sin\theta|},\ 
	\dfrac{l\sin\theta}{|\kappa_4 \sin\theta|},\ 
	0,\ 
	0
	\right).
\end{align}
From these data, we can derive striction surface and curve of $S_4$,
and the conditions for singularities by 
\eqref{eq:str1}, \eqref{eq:str2} and Theorem \ref{thm:singcond},
respectively.
We omit the detail.

\medskip
{\footnotesize
\begin{flushright}
\begin{tabular}{l}
Department of Mathematics,\\
Graduate School of Science, \\
Kobe University, \\
Rokkodai 1-1, Nada, Kobe \\
657-8501, Japan\\
E-mail: {\tt wayne990203@outlook.com}\\
E-mail: {\tt saji@math.kobe-u.ac.jp}
\end{tabular}
\end{flushright}}

\end{document}